\documentclass[12pt,letterpaper,reqno]{amsart}

\usepackage{amsthm}
\usepackage{mathrsfs}
\usepackage{amsfonts,amssymb,amsmath}
\input amssym.def 
\input amssym.tex

\numberwithin{equation}{section}

\usepackage[dvips]{graphicx}
\usepackage{hyperref}
\usepackage{thmtools}

\declaretheoremstyle[
bodyfont=\normalfont,
]{remstyle}

\declaretheorem[style=remstyle, name=Remark,numberwithin=section]{rem}
\declaretheorem[style=remstyle, name=Question]{question}

\newtheorem{defi}{\bf Definition}[section]

\newtheorem{conj}{\bf Conjecture}[section]
\newtheorem{cond}{\bf Condition}[section]

\newtheorem{lemma}{\bf Lemma}[section]

\newtheorem{theorem}{\bf Theorem}[section]

\newtheorem{corr}{\bf Corollary}[section]
\newtheorem{corollary}{\bf Corollary}[section]

\newcommand\be{\begin{eqnarray*}}
\newcommand\ee{\end{eqnarray*}}
\newcommand\beq{\begin{equation}}
\newcommand\eeq{\end{equation}}
\newcommand\eps{\epsilon}

\newcommand{\EXP}{\mathbb{E}}

\newcommand\ben{\begin{eqnarray}}
\newcommand\een{\end{eqnarray}}

\begin{document}

\title[On sets with small additive doubling in product sets.] 
{On sets with small additive doubling in product sets.}

\author{Dmitry Zhelezov}
\thanks{Department of Mathematical Sciences, 
Chalmers University Of Technology and University of Gothenburg} 
\address{Department of Mathematical Sciences, 
Chalmers University Of Technology and University of Gothenburg,
41296 Gothenburg, Sweden} \email{zhelezov@chalmers.se}


\date{\today}

\begin{abstract}
	Following the sum-product paradigm, we prove that for a set $B$ with polynomial growth, the product set $B.B$ cannot contain large subsets with size of order $|B|^2$ with small doubling. It follows that the additive energy of $B.B$ is asymptotically $o(|B|^6)$. In particular, we extend to sets of small doubling and polynomial growth the classical Multiplication Table theorem of Erd\H{o}s saying that $|[1..n]. [1..n]| = o(n^2)$. 
\end{abstract}

\maketitle

\section{Introduction}
  The famous sum-product conjecture of Erd\H{o}s and Szemer\'edi \cite{ES} states that for any $\epsilon > 0$ and there is a constant $c(\epsilon)$ such that for arbitrary of real numbers $B$ holds
  $$
  \max(|B.B|,|B+B|) \geq c|B|^{2-\epsilon},
 $$
 where $B.B = \{ bb'|\,\,b, b' \in B\}$ and $B + B = \{b + b'|\,\,b, b' \in B\}$. The former set is called the \emph{product} set of $B$ and the latter is the \emph{sumset}.

  The intuition behind this conjecture is that there are no approximate subrings in the set of reals, that is, a set must expand either with respect to addition or multiplication. However, it might be almost closed with respect to, say, only addition, so that $|B+B| < K|B|$. In this extreme case the inverse Freiman theorem tells that $B$ has a very rigid structure, namely that it looks like a generalized arithmetic (geometric, if the the set does not expand with respect to multiplication) progression.  Another important and active area of research in additive combinatorics is to extract structural information about sets which are sumsets or product sets (see, e. g. \cite{Shkr14} for one of the most recent results and references therein).  
  
  Combining these two lines of inquiry, it seems reasonable to ask if an arbitrary sumset (product) set can be multiplicatively (resp. additively) small, so a large portion of it looks like a geometric (resp. arithmetic) progression. 

  Some results were obtained in this direction. The author \cite{Zh14} proved that for sets of integers if $A = B.B$ with $|B| = n$, then the size of any arithmetic progression contained in $A$ is bounded by $O(n \log n)$, and this bound is sharp up to a multiplicative constant. Senger \cite{Senger} has  shown that the additively shifted product set $B.B + 1$ must have a large part outside of any generalized geometric progression of comparable length. In seems natural to consider a more general setting, namely the following questions. Let $B$ be a set of size $n$ and $K > 0$ be fixed.  
\begin{question}
How large can a set $A \subset B+B$ be if $|A.A| < K|A|$?
\end{question}
\begin{question}
How large can a set $A \subset B.B$ be if $|A+A| < K|A|$?  
\end{question}

It is worth noting that Question 1 is closely related to the famous Erd\H{o}s unit distance conjecture, which asserts that a set of $n$ points on the Euclidean plane defines at most $O(n^{1+o(1)})$ unit distances. Indeed, if we identify the set of points with a set of complex numbers $B$, then the Erd\H{o}s unit distance conjecture is equivalent to the following problem.

\begin{conj}[Erd\H{o}s unit distance conjecture] \label{conj:erdos}
	Let $B$ be a set of complex numbers of size $n$. Then
	$$
		\#\{ (z_1, z_2) \in B \times B |\,\, z_1 - z_2 \in \mathbb{S}^1 \} = O(n^{1+o(1)}),	
	$$
	where $\mathbb{S}^1$ is the unit circle. 
\end{conj} 

Of course, it is doesn't matter whether we count differences $z_1 - z_2$ or sums $z_1 + z_2$ since we can always consider another set 
$\{ B \cup -B \}$  sacrificing just a constant factor.

\begin{rem}
	Both Question 1 and Question 2 may be posed in a slightly more general setting akin to Conjecture \ref{conj:erdos}. Instead of bounding the size of  $A$ one might be interested in the number of \emph{pairs} $(b_1, b_2)$ s.t. $b_1 + b_2 \in A$ (resp. $b_1b_2 \in A$ for Question 2). We will call such a formulation \emph{counting with multiplicities}.
\end{rem}

\vspace{0.5em}
Conjecture \ref{conj:erdos} was verified by Schwartz, Solymosi and de Zeeuw \cite{SSZ} with $\mathbb{S}^1$ replaced with the roots of unity and by Schwartz \cite{S} in the case when $\mathbb{S}^1$ is restricted to a multiplicative subgroup of finite rank.  In turn, the following lemma was proved by Roche-Newton and the author  in order to estimate the additive energy of a sumset.

\begin{lemma}[\cite{OllyMe}] \label{lemma:sumset_multp_group}
	Let $\epsilon > 0$. Then there are positive constants $c(\epsilon), C(\eps)$ such that for any set $B$ of complex numbers the following holds. For any multiplicative group $\Gamma \subset \mathbb{C}^*$ of rank $c(\eps)\log |B|$, the number of pairs
	$$
		\{ (b_1,b_2) \in B \times B \,\,| \,\, b_1-b_2 \in \Gamma \}
	$$ 
is at most $C(\eps)|B|^{1+\epsilon}$.	
\end{lemma}

This lemma actually answers Question 1. Indeed, by the inverse Freiman theorem, $|A.A| < K|A|$ implies that $A$ is contained in a multiplicative subgroup $\Gamma$ of rank at most $K$ and then it follows that $|A| = |B|^{1+o(1)}$. We refer the reader to \cite{OllyMe} for details\footnote{For example, one can see that in fact even if $K$ is as large as $c(\epsilon) \log |A|$ the bound $|A| = O(|B|^{1+\epsilon})$ holds.}.

Despite the fact that Question 2 looks similar to Question 1 at a first glance, we were not able to give such a sharp bound as in Lemma \ref{lemma:sumset_multp_group}, mainly due to the lack of a suitable replacement of the Subspace theorem, used in \cite{OllyMe}. To convince the reader that the question is probably subtle, let us take $B = [n]$, the integers from $1$ to $n$, and let $A = B.B$.  Then any non-trivial answer to Question 2, i.e. even the bound $|A| = o(|B|^2)$, already implies Erd\H{o}s' Multiplication Table theorem, as explained below.

This observation indicates that perhaps the most natural setting to start with is the integer case.  The main part of this note is to prove the following theorem with the following technical condition.
\begin{cond}[Polynomial growth] \label{cond:polygrowth}
	Let $C_1, C_2 > 0$ be fixed absolute constants. A set of integers $B$ is of \emph{polynomial growth} if for any $b \in B$ holds
	$$
		|b| < 	C_1|B|^{C_2}.
	$$
\end{cond}

\begin{theorem} \label{thm:smalldoubling}
	Let $K > 0$ be fixed and $B$ be a set of integers of polynomial growth. If $A \subset B.B$ and $|A+A| < K|A|$ then $|A| = o(|B|^2)$\footnote{To be precise, here we assume that we have a sequence of sets $B_k$ of polynomials growth such that $|B_k| \to \infty$. Then Theorem \ref{thm:smalldoubling} says that $|A|/|B_k|^2 \to 0$. Later on we will always silently assume that the set $B$ is "large" in the aforementioned sense. See also the notation section.}.
\end{theorem}

Unfortunately, at the moment we don't know how to get rid of Condition \ref{cond:polygrowth} nor we can extend Theorem \ref{thm:smalldoubling} to sets of real numbers. Nevertheless, the Multiplication Table theorem follows as a corollary unconditionally.
\begin{corr}[Erd\H{o}s, \cite{ErdMult}]
Let $B=[n]$. Then $|B.B| = o(n^2)$.
\end{corr}
\begin{proof}
Indeed, assume for contradiction that $|B.B| \geq cn^2$ for some absolute $c > 0$. Take $A = B.B$, hence $A + A \subset [2n^2]$, and we have
$$
\frac{|A+A|}{|A|} < \frac{2}{c}.
$$  
But this contradicts Theorem \ref{thm:smalldoubling}, since clearly $B$ is of polynomial growth.
\end{proof} 
 
Recall that the \emph{additive energy} $E_+(A)$ of a set $A$ is defined as the number of quadruples $(a_1, a_2, a_3, a_4)$ such that 
$a_1 + a_2 = a_3 + a_4$. We will often also say that $A$ has small doubling if $|A+A| \ll |A|$. The additive energy of sets of size $n$  with small doubling is of order $n^3$, but the opposite is not true. However, the Balog-Szemer\'edi-Gowers theorem (see \cite{TV}) states that if $A$ is of size $n$ and the additive energy $E_{+}(A)$ is of order $n^3$, one can find a subset $A' \subseteq A$ with $|A'| \gg n$ such that $A'$ has small doubling. The multiplicative energy $E_{\times}(A)$ is defined in a similar way as the number of 4-tuples $(a_1, a_2, a_3, a_4) \in A^4$ such that $a_1a_2 = a_3a_4$   .

Thus, the Balog-Szemer\'edi-Gowers theorem provides an immediate corollary of Theorem \ref{thm:smalldoubling} (which is in fact equivalent to it).

\begin{corr}
	For an integer set of polynomial growth $B$ holds
	$$
		E_{+}(B.B)	= o(|B|^6).
	$$
\end{corr}


 
The rest of the paper is devoted to the proof of Theorem \ref{thm:smalldoubling}. The proof is rather lengthy and will consist of several steps. We start with the contrapositive assumption that there is a set $A$ with small doubling and $\Omega(|B|^2)$ pairs $(b_1, b_2)$ s.t. $b_1b_2 \in A$. Refining the structural information about $A$ and $B$ we will eventually conclude that one can replace $A$ and $B$ with generalized arithmetic progressions with certain properties. The final step is to rule out this possibility and the argument is more in the spirit of Erd\H{o}s' original proof of the multiplication table theorem (which now follows as a special case). This is the only place where the polynomial growth condition is used and we strongly believe that Theorem \ref{thm:smalldoubling} in fact holds unconditionally and also for arbitrary sets of complex numbers (since the rest of the arguments go through). 
 

%

\section{Notation and Definitions}

Let $f,g : \mathbb{N} \rightarrow \mathbb{R}_{+}$. 
The following standard notation will be used in this paper: 
\begin{enumerate}

\item $f(n) = O(g(n))$ means that $\limsup_{n \rightarrow \infty} \frac{f(n)}{g(n)} < \infty$. It is equivalent to $f(n) \ll g(n)$ or $g(n) \gg f(n)$.
\item $f(n) = \Omega(g(n))$ means that $g(n) = O(f(n))$.
\item $f(n) = o(g(n))$ means that $\lim_{n \rightarrow \infty} \frac{f(n)}{g(n)} = 0$. 
\item $f(n) \lll g(n)$ means $f(n) = o(g(n))$ and $f(n) \ggg g(n)$ is equivalent to $g(n) \lll f(n)$.

\end{enumerate}

The number $n$ is always assumed to be large being the size of our set $B$. It is helpful to define a graph which represents the way $A$ is contained in $B.B$.  

\begin{defi}
   Let $A \subset B.B$. The \emph{containment graph with multiplicities} $G^m(A, B.B)$ is a bipartite graph on $(B, B)$ s.t. $(b_1, b_2)$ are adjacent iff $b_1 \leq b_2$ and $b_1b_2 \in A$. Let us leave out some edges of $G^m$ such that for each $a \in A$ there remains only the first pair  $(b_1, b_2)$ with $b_1b_2 = a$, in lexicographical order\footnote{Of course, the actual ordering does not matter.}. We will call the remaining graph $G(A, B.B)$ the \emph{containment graph}.
\end{defi}

The notion of the containment graph (with multiplicities) is similar to the unit distance graph usually defined for the Erd\H{o}s unit distance problem. We will often write $G^m$ or $G$ when it is clear which sets $A$ and $B$ are meant. 

Let us now record our assumptions about $A$ and $B$. Let $K > 0, \alpha > 0$ be fixed constants and $|B| = n$. In what follows we will assume for contradiction that $|A+A| < K|A|$ and $E(G_m) \geq \alpha n^2$. Every time we use the $\ll$ or $\gg$ notation we silently assume that the implicit constant depends on $K$ and $\alpha$ only. All explicit constants may depend on $K$ and $\alpha$ as well. Also, we say that a set $X$ is \emph{dense} in $Y$ if 
$|Y \cap X| \gg |Y|$. A graph $G$ is said to be \emph{dense} if $|E(G)| \gg |V(G)|^2$. A set $U$ is said to have \emph{small doubling} if $|U + U| \ll |U|$. Again, we assume that the sets in question are large while the implicit constant does not depend on the sizes of the sets.

In due course we will update $A$ and $B$ which may worsen  $K$ and $\alpha$ only by a constant factor, so we will use the same letters, slightly abusing notation.

Sometimes instead of taking the sumset $U+W$ we will take a restricted sumset $U \stackrel{G'}{+} W$ along a bipartite graph $G'$ with the color classes $(U, W)$. It is defined as
$$
	U \stackrel{G'}{+} W = \{ u+v  \,|\, (u, v) \in E(G')\}.
$$
A restricted partial product set along a graph is defined similarly. We will also use repeated sumsets $\ell A = A + \ldots + A$, where $A$ is taken $\ell$ times.

A \emph{generalized arithmetic progression} (or GAP) is a set of the form
$$
	P = \{a + d_1x_1 + \dots d_kx_k : 0 \leq x_i \leq L_i \},
$$
where $k$ is the \emph{rank}, $Vol(P) = \prod_{i=1}^k L_i$ is the \emph{volume} spanned by the $k$-dimensional box and $d_1, \ldots, d_k$ are the \emph{differences}.  A GAP is \emph{proper} if all its elements are distinct, i.e. $|P| = \prod_{i=1}^k (L_i + 1)$.
 
The following, now classical, theorem due to Freiman and Ruzsa (see e.g. \cite{TV}, which is a standard textbook reference in additive combinatorics) is crucial for our future arguments and asserts that sets with small doubling have very rigid structure.

\begin{theorem}[Freiman-Ruzsa inverse theorem] \label{thm:Freiman-Ruzsa}
	Let $K>0$ be fixed. If $|A+A| < K|A|$, then $A$ is contained in a proper GAP of rank at most $K$ and size at most $c(K)|A|$.
\end{theorem}  

The problem of estimating the bounds in Theorem \ref{thm:Freiman-Ruzsa} has received considerable attention. However, since we assume that $K$ is constant and fixed once and for all, the dependence of $c$ on $K$ is not important for us. 

Another standard result is the Pl\"unnecke-Ruzsa inequality. Again we refer the reader to the book \cite{TV} and a more recent and elegant proof due to Petridis \cite{Petr}.

\begin{theorem}[Pl\"unnecke-Ruzsa inequality] \label{thm:Plunecke-Ruzsa}
	Let $A$ be an additive set  with $|A + A| ≤ K |A|$. Then for all nonnegative integers $m, n$ we have
$$
|mA - nA| \leq K^{m+n}|A|
$$
\end{theorem}




\section{Eliminating multiplicities}

In this section we want to show that if $G^m$ is dense then the graph $G$ must be also dense, so that one can assume that every element in $A$ is counted only once.

To analyze $G$ we will use a well-known result about dense bipartite graphs. Write $N_G(v)$ for the set of neighbours of a vertex $v$ in a graph $G$ (the subindex $G$ will be often omitted). The following fact about dense bipartite graphs  was  probably used for the first time by Gowers \cite{Gow} to improve bounds in the Balog-Szemer\'edi theorem.

\begin{lemma} \label{lm:pairs}
Let $G$ be a bipartite graph on  $(B_1, B_2)$ where $|B_1| = |B_2| = n$ and $0 < \alpha = E(B_1, B_2)/n^2$. Let $0 < \epsilon < 1$ be fixed. Then there is a subset $B'_1 \subseteq B_1$ with $|B'_1| \geq \alpha n/2$ such that for at least $(1-\epsilon)|B'_1|^2$ of the ordered pairs of vertices $(v_1, v_2) \in B'_1 \times B'_1$ holds 
$$
|N(v_1) \cap N(v_2)| \geq \frac{\epsilon \alpha^2 n}{2}.
$$
\end{lemma} 

\begin{lemma} \label{lm:largeA}
	If $G^m$ is dense then so is $G$.
\end{lemma}

\begin{proof}
	Let $G^m$ be the graph defined above and $\alpha = |E(G^m)|/n^2$. Fix $\epsilon > 0$ to be defined later. By Lemma \ref{lm:pairs} there is a set $B_1$ with $|B_1| \geq \alpha n/2$ such that a proportion $(1-\epsilon)$ of the pairs of vertices in $B_1$ share at least $\epsilon \alpha^2 n/2$ neighbours. Let $G' \subseteq B_1 \times B_1$ be the set of such pairs sharing a large number of common neighbors. 
	
	Now we are going to define a set of lines in the plane and then apply the Szemer\'edi-Trotter incidence theorem (another standard tool in additive combinatorics, see \cite{TV}). Define the set of lines 
	$L = \{l_{ijk}: y = (b_i + b_j)x + a_k | (b_i, b_j) \in G', a_k \in A \}$ and points $P = B \times (A + A + A)$ so $|P| \ll n|A|$ by the Pl\"unnecke-Ruzsa inequality.  Each line $l_{ijk}$ is incident to at least $\epsilon \alpha^2 n/2$ points in $P$ with $x$-coordinates in $X_{ij} = N(b_i) \cap N(b_j)$, by our choice of $G'$.
	
	Applying the Szemer\'edi-Trotter theorem, we have
 \beq \label{eq:Sz-Tr}
	 |L| = |B_1 \stackrel{G'}{+} B_1||A| \ll \frac{n^2|A|^2}{\epsilon^3 n^3} + \frac{|A|}{\epsilon}.
 \eeq
	Using the trivial bound $|B_1 \stackrel{G'}{+} B_1| \geq (1-\epsilon)|B_1|$ with, say, $\epsilon = 1/4$ we obtain
	$$
	   n^2 \ll n|B_1| \ll |A|.
	$$
\end{proof}

\begin{corollary} \label{corr:pairs_imply_small_doubling}
	Let $V \subset B$ with $|V| \gg |B| = n$. Let $\epsilon > 0$ be fixed and $G'$ be the set of pairs $(v_1, v_2) \in V \times V$ s.t. 
	$$
		|N(v_1) \cap N(v_2)| \gg_\epsilon n.
	$$ 
	Then $|V \stackrel{G'}{+} V| \ll_\epsilon n$. 
\end{corollary}
\begin{proof}
	Arguing as in Lemma \ref{lm:largeA}, we have by (\ref{eq:Sz-Tr})
	$$
	|V \stackrel{G'}{+} V| \ll_{\epsilon} \frac{|A|}{n} \ll_{\epsilon} n .
	$$
\end{proof}

From now on we will work only with the graph $G$, so that $E(G) = |A|$.

\section{Finding GAPs}

 By Corollary \ref{corr:pairs_imply_small_doubling} if we take a set $B_1$ such that a large portion of pairs of vertices in $B_1$ share a lot of neighbors, it is additively small along a dense graph. Thus, we can apply the Balog-Szemer\'edi-Gowers theorem and find a dense subset with small doubling inside $B_1$. In this section we want to achieve a bit more, namely find two dense subsets $V, W$ inside $B$ with small doubling, such that $G$ restricted to $(V, W)$ is still dense.

We need the following BSG-type result for very dense graphs, see \cite{TV}, exercise 2.5.4.
\begin{lemma} \label{lm:denseBSG}
	Let $0 < \epsilon < 1/4$ and let $G \subseteq A\times B$ with $|G| = (1 - \epsilon)|A||B|$. Then there exists $A' \subseteq A$ 
	with $|A'| = (1 -\sqrt{\epsilon})|A|$ such that
	$$
		|A' + A'| \leq \frac{|A\stackrel{G}{+}B|^2}{(1-2\sqrt{\epsilon})|A|}
	$$
\end{lemma}

\begin{lemma} \label{lem:small_doubling}
	Let $G(A, B.B)$ be dense. Then there are sets $V, W \subseteq B$ with $|V|, |W| \gg |B| = n$ such that both $V$ and $W$ have small doubling and $G$ restricted to $(V, W)$ is dense.
\end{lemma}
\begin{proof}
	The proof loosely follows the graph-theoretic part of Gowers' proof of the Balog-Szemer\'edi-Gowers theorem. 
	
	Let $B_1, B_2$ be the color classes of $G$ and $\alpha = |E(G)|/n^2$.	We start with deleting all edges emanating from vertices in $B_1$ with degree less than $\alpha n/2$. Since the total number of edges is $\alpha n^2$, at least $\alpha n^2/2$ edges will remain. Let  $\alpha'$ be the edge density of this new graph. Fix $0 < \epsilon < 1$ to be chosen later. By Lemma \ref{lm:pairs} there is a subset $B'_1$ of size $\alpha' n/2$ such that a proportion of $(1-\epsilon)$ of pairs of $B'_1$ have at least $\epsilon \alpha'^2 n/2$ common neighbours. Thus, if $\epsilon < 1/4$ at most half of the vertices of $B'_1$ can have degree $0$, and the others have degree at least $\alpha n/2$ by our first deletion step. We remove these zero-degree vertices, clearly the proportion of pairs with large common neighborhoods can only increase by this removal. Let us denote by the same letter $B'_1$ the resulting set and record that $|B'_1| \geq \alpha' n/4 \geq \alpha n/8$ and each vertex in $B'_1$ has degree at least $\alpha n/2$. 
	
	Let $B'_2$ be the set of neighbours of $B'_1$. Since $E(B'_1, B'_2) \geq \alpha^2 n^2/16$, it follows that at least $\alpha^2 n/32$ vertices of $B'_2$ have at least $\alpha^2 n/32$ neighbours in $B'_1$. Let us reassign $B'_2$ to this subset of rich vertices and $\alpha''$ be the density of $G$ restricted to $(B'_1, B'_2)$. Clearly, $\alpha'' \geq \alpha^2/32$.
	
	Now we again apply Lemma \ref{lm:pairs} to $(B'_1, B'_2)$  to get a subset $B''_2 \subseteq B'_2$ with $|B''_2| \geq \alpha'' n/2$ such that at least $(1-\epsilon)$ of all pairs of vertices of $B''_2$ share at least $(\alpha'')^2 \epsilon n/2$ neighbours. Since each vertex of $B''_2$ has at least $\alpha^2 n/32$ neighbours in $B'_1$ we also have
	$$
	E(B'_1, B''_2) \geq \frac{|B''_2|\alpha^2 n}{32} \geq \frac{\alpha'' \alpha^2 n^2 }{64} \geq \frac{\alpha^4 n^2}{2^{11}}.
	$$
	
	Let $H_1$ be the set of pairs in $B'_1$ such that all pairs in $H_1$ have at least $\alpha'^2\epsilon  n/2 $ common neigbours in $G$ and $H_2$ be the set of pairs in $B''_2$ with at least $\alpha''^2\epsilon n/2$ common neighbours. We know that 
	$|H_1| \geq (1-\epsilon)|B'_1|^2$ and $|H_2| > (1-\epsilon)|B''_2|^2$.
	
   Take $\epsilon =  \alpha^8/2^{30} < \frac{1}{4}$. By Lemma \ref{lm:denseBSG}, there are subsets $B''_1 \subset B'_1, B'''_2 \subset B''_2$ with $|B''_1| \geq (1-\sqrt{\epsilon})|B'_1|$ and $|B'''_2| \geq (1-\sqrt{\epsilon})|B''_2|$ such that
\be
	|B''_1 + B''_1| &\ll& \frac{|B'_1\stackrel{H_1}{+}B'_1|^2}{|B'_1|} \\
	|B'''_2 + B'''_2| &\ll& \frac{|B''_2\stackrel{H_2}{+}B''_2|^2}{|B''_2|}.
\ee	

	On the other hand, by Corollary \ref{corr:pairs_imply_small_doubling} both RHS are $\ll n$ and since $B''_1$ and $B'''_2$ are dense in $B$ they have small doubling. On the other hand,
	$$
	E(B''_1, B'''_2) \geq E(B'_1, B''_2) - (\epsilon + 2\sqrt{\epsilon})|B'_1||B''_2| \gg n^2,
	$$
	so a positive proportion of edges of $G$ lie between $B''_1$ and $B'''_2$.
	 
\end{proof}

The final step in this section is to apply Theorem \ref{thm:Freiman-Ruzsa} to the sets $V, W$ provided by Lemma \ref{lm:denseBSG} and $A' = \{ uv |\,\, (u, v) \in E(V, W) \}$ which all have small doubling. We record the following corollary of Lemma \ref{lm:denseBSG}.

\begin{corollary} \label{corr:dense_subsets_of_GAPs}
Let $G(A, B.B)$ be dense. Then there are proper GAPs $P_1, P_2, P_3$ with ranks bounded by $O(1)$ and the following properties
\begin{enumerate}
\item $n \ll |P_1|, |P_2| \ll n$.
\item There is a dense graph $G'$ on $(P_1, P_2)$ such that $\{ P_1 \stackrel{G'}{\cdot} P_2 \}$ is a dense subset of $P_3$.
\item $n^2 \ll P_3 \ll n^2$.
\end{enumerate}

\end{corollary}

\section{Divisor statistics} \label{section:divisor_statistics}

Corollary \ref{corr:dense_subsets_of_GAPs} gives us dense subsets of two proper GAPs $P_1, P_2$ of size $\Omega(n)$ such that a large subset of $P_1P_2$ is dense inside another GAP $P_3$ of bounded rank. 

Now we invoke Condition \ref{cond:polygrowth} (and this is the only place we use it), to claim that the differences of $P_1$ and $P_2$ are bounded by $n^{O(1)}$. Clearly, the differences of $P_3$ are then polynomially bounded as well.  Let $M_1, M_2, M_3$ be the sizes of the longest subprogressions of $P_1, P_2$ and $P_3$ respectively and $D_1, D_2, D_3$ be corresponding differences. By the polynomial growth condition, $D_i < n^{O(1)}$, so the $D_i$s have at most $o(\log n)$ distinct prime divisors. Also, $M_i > n^\delta$ where delta depends only on the ranks. Let $P$ be the set of primes $p$ such that $(p, D_i) = 1$. We are now going to implement a strategy similar to the original proof of the Multiplication Theorem due to Erd\H{o}s, closely following the exposition by Croot \cite{Croot}.

We write 
\ben
P_1 &=& \{d_0 + d_1[I_1] + \ldots + d_k[I_k] \} \\
P_2 &=& \{e_0 + e_1[J_1] + \ldots + e_l[J_l] \} . 
\een
Without loss of generality, assume $D_1 = d_1$, $D_2 = e_1$, $M_1 = I_1$, $M_2 = J_1$. Let $\Omega_{P, N}(x)$ denote the number of prime power divisors $p^a$ of $x$ such that $p \in P$ and $p^a \leq N$.

It turns out that even restricted to certain primes, $\Omega$ is concentrated around its mean.
\begin{lemma} \label{lemma:divisor_statistics}
	Let $\delta > 0$ and $P_1$ be a proper GAP with $|P_1| = n$ such that $M_1 > n^\delta$, $D_1 = n^{O(1)}$ and $P$ is the primes with at most $o(\log n)$ elements excluded. Then for all but at most $o(n)$ elements $x \in P_1$ holds
	$$
		\log \log n -  (\log \log n)^{2/3} < \Omega_{P, n^\delta}(x) < 	\log \log n + (\log \log n)^{2/3}
	$$
\end{lemma}
\begin{proof}
Let $X$ be an element of $P_1$ selected uniformly at random and let $Y = \Omega_{P, n^\delta}(X)$. We have
\ben
\EXP(Y) &=& \sum_{x \in P_1}  \Omega_{P, n^\delta}(X) \mathbb{P}( X= x) \\ \label{eq:expectation}
        &=& \frac{1}{n} \sum_{x \in P_1} \sum_{\substack{p^a | x \\ p \in P \\ p^a \leq n^\delta}} 1 = \sum_{\substack{p \in P \\ p^a \leq n^\delta}} \frac{1}{n}  \#\{x \in P_1 |\,\, p^a | x\}. 
\een
\end{proof} 

We claim that 
\beq \label{eq:prime_powers_bound}
\#\{x \in P_1 |\,\, p^a | x\} = \frac{n}{p^a} + O(n^{1-\delta})
\eeq
for any $p \in P$. Indeed, since $(d_1, p^a) = 1$, for fixed $i_2, \ldots, i_k$ there are at least $I_1/p^a - 1$ and at most $I_1/p^a + 1$ integers $i_1 \in [I_1]$ such that $p^a | (d_0 + d_1i_1 + \ldots + d_ki_k)$. Thus, since $P_1$ is proper,
$$
\#\{x \in P_1 |\,\, p^a | x\} = \frac{I_1I_2 \ldots I_k}{p^a} + O(I_2 \ldots I_k) = \frac{n}{p^a} + O(n^{1-\delta}).
$$
 
Plugging into (\ref{eq:expectation}) we obtain 
$$
\EXP(Y) = \sum_{\substack{p \in P \\ p^a \leq n^\delta}} \frac{1}{p^a} + O(1).
$$

Using the well-known fact that 
$$
\sum_{\substack{p^a \leq n \\ p \text{ prime }}} \frac{1}{p^a} = \log \log n + O(1)
$$ 
and that $P$ misses at most $o(\log n)$ primes we conclude that
$$
\EXP(Y) = \log \log n - o(\log \log \log n).
$$ 

It remains to estimate the variance $V(Y)$. Recall 
$$
V(Y) = \EXP(Y^2) - \EXP(Y)^2 = \EXP(Y^2) - (\log \log n - o(\log \log \log n))^2.
$$

We have
$$
E(Y^2) = \frac{1}{n}\sum_{x \in P_1} (\sum_{\substack{p \in P \\ p^a \leq n^\delta \\ p^a | x}} 1)^2 = \frac{1}{n} \sum_{x \in P_1} \sum_{\substack{p^a | x, q^b | x\\ p^a, q^b \leq n^\delta \\ p, q \in P}} 1 = \frac{1}{n}  \sum_{\substack{p^a, q^b \leq n^\delta \\ p, q \in P}} \sum_{\substack{x \in P_1 \\ p^a | x, q^b | x }}1
$$

Arguing as before, we have that if $p$ and $q$ are distinct then $p^aq^b$ divides at most $n/(p^aq^b)$ elements $x \in P_1$. If $p = q$ and $2 \leq a \leq b$ then the number of such elements is bounded from above by $n/p^b$ and in this case we have
$$
\sum_{\substack{p^a, p^b \leq n^\delta \\ a \leq b \\ p, q \in P}} \frac{1}{p^b} \leq \sum_{\substack{p^b \leq n^\delta \\ b \geq 2 }} 
\frac{b}{p^b}  \leq \sum_{\substack{p^b \leq n^\delta \\ b \geq 2 }} \frac{\log_2 (p^b)}{p^b} = O(1).
$$   

Finally, the remaining term with $p = q$ and $a=b=1$ is just $\EXP(Y) = \log \log n - o(\log \log \log n)$.

Thus, we obtain 

\be
\EXP(Y^2) &=& O(1) + \EXP(Y) + \sum_{\substack{p^a, q^b \leq n^\delta \\ p \neq q}} \frac{1}{p^aq^b} \\
                  &\leq& (\log \log n + O(1))^2 + \log \log n - o(\log \log \log n),
\ee
and therefore
$$
V(Y) = O(\log \log n \log \log \log n).
$$

By the Chebyshev inequality it now immediately follows that
$$
\mathbb{P}(|Y - \EXP(Y)| \geq (\log \log n)^{2/3} ) \leq O((\log \log n)^{-1/3} \log \log \log n)
$$
which implies Lemma \ref{lemma:divisor_statistics} since $\EXP(Y) = \log \log n - o(\log \log \log n)$.

\begin{rem}
	The proof Lemma \ref{lemma:divisor_statistics} works just as well if we allow $P$ to miss more than $o(\log n)$ which in turn would imply that the sizes of the elements of our sets can grow even faster than polynomially. Since this improvement is not substantial, we have decided to go with a more strict condition for brevity.  
\end{rem}

\section{Putting everything together}   

   It is now easy to reach the desired contradiction and finish the proof of Theorem~\ref{thm:smalldoubling} in the case where sizes of the elements in $B$ are bounded by $|B|^C$ for some $C > 0$. Assume that the claim of the theorem does not hold. Then by 
   Lemmas~\ref{lm:largeA} to \ref{lemma:divisor_statistics} we conclude that there must exist two proper GAPs $P_1, P_2$ such that a dense subset of their product set has small doubling and therefore dense inside another proper GAP $P_3$. We now pick $\delta > 0$ such that the largest subprogression of $P_1,P_2$ and $P_3$ has length at least $n^\delta$ and apply Lemma  \ref{lemma:divisor_statistics} to $P_1, P_2$ with $\delta/2$ and to $P_3$ with $\delta$ and the set of primes $P$ as defined at the beginning of Section~\ref{section:divisor_statistics}. Clearly $\Omega_{P, n^{\delta/2}}(x) + \Omega_{P, n^{\delta/2}}(y) \leq \Omega_{P, n^{\delta}}(xy)$ so, for at least $\Omega(|P_1||P_2|)$ pairs $x \in P_1, y \in P_2$ we have 
   $$
    2 \log \log n - O((\log \log n)^{2/3}) \leq \Omega_{P, n^{\delta}}(xy) 
   $$ 
   But the set of such elements cannot be dense in $P_3$ since for at least $|P_3| - o(|P_3|)$ elements $z \in P_3$ holds 
   $$
   \Omega_{P, n^{\delta}}(z) \leq  \log \log |P_3| + O((\log \log |P_3|)^{2/3})
   $$
   and we arrive at a contradiction. 
   
   \qed

\section{Acknowledgements}
  I am grateful to Peter Hegarty for helpful discussions and careful proofreading of the manuscript.

\end{document}